\documentclass[10pt]{amsart}

\usepackage[T1]{fontenc}
\usepackage{lmodern}
\usepackage{amsmath,amssymb,amsfonts,mathtools}
\usepackage{enumitem}
\usepackage{microtype}
\usepackage[hidelinks]{hyperref}
\hypersetup{
  pdftitle={Algebraic values of transcendental functions with geometric coefficient moduli},
  pdfauthor={Diego Marques},
  pdfkeywords={Mahler's Problem A, algebraic values, transcendental functions, geometric coefficient moduli, polygonal cancellation, non-Archimedean valuations, $S$-units}
}

\newtheorem{theorem}{Theorem}[section]
\newtheorem{proposition}[theorem]{Proposition}
\newtheorem{lemma}[theorem]{Lemma}
\newtheorem{corollary}[theorem]{Corollary}
\newtheorem*{questionA}{Mahler's Problem A}
\theoremstyle{definition}
\newtheorem{definition}[theorem]{Definition}
\theoremstyle{remark}
\newtheorem{remark}[theorem]{Remark}

\numberwithin{equation}{section}

\newcommand{\N}{\mathbb N}
\newcommand{\Z}{\mathbb Z}
\newcommand{\Q}{\mathbb Q}
\newcommand{\R}{\mathbb R}
\newcommand{\C}{\mathbb C}
\newcommand{\Qbar}{\overline{\mathbb Q}}
\newcommand{\D}{\mathcal D}
\newcommand{\Rset}{\mathcal R}
\newcommand{\abs}[1]{\lvert #1\rvert}
\newcommand{\ord}{\operatorname{ord}}
\newcommand{\B}{B(0,1)}

\title[Transcendental power series with geometric coefficient moduli]
{Algebraic values of transcendental power series with geometric coefficient moduli}

\author{Diego Marques}
\address{Departamento de Matem\'atica, Universidade de Bras\'ilia, Bras\'ilia, DF, Brazil}
\email{diego@mat.unb.br}


\begin{document}

\begin{abstract}
Let $\lambda>1$ be a real algebraic number. We construct continuum many
power series $f(z)=\sum_{k\geq0}a_kz^k$ of radius of convergence exactly
one such that every nonzero coefficient $a_k$ is algebraic and has modulus
$\lambda^m$ for some $m\geq0$. Moreover, for every integer $s\geq0$, the derivative $f^{(s)}$ takes
algebraic values at all algebraic points of the open unit disk and is
transcendental over $\C(z)$. The proof combines algebraic polygonal cancellation with a sparse
polynomial-block argument. This shows that a multiplicative rank-one
restriction on coefficient moduli is compatible with algebraicity of
the full analytic jet at every algebraic point once algebraic phases
are allowed.
\end{abstract}

\maketitle

\section{Introduction}

The arithmetic nature of the values of transcendental functions has been a
classical theme since the nineteenth century. By a \emph{transcendental
function} on a domain $\Omega\subseteq\C$, we mean an analytic function that
is not algebraic over $\C(z)$; equivalently, it satisfies no identity
$P(z,f(z))=0$ with $0\neq P\in\C(z)[Y]$. Although classical transcendence
theorems often force transcendental values at algebraic arguments, no such
principle holds for arbitrary transcendental functions.

In an 1886 letter to Strauss, Weierstrass asked whether there exists a
transcendental entire function taking algebraic values at every algebraic
point. St\"ackel answered this question affirmatively in 1895 and proved a
much more flexible interpolation theorem: for a countable set
$\Sigma\subseteq\C$ and a dense set $T\subseteq\C$, one can construct a
transcendental entire function $f$ such that
\[
        f(\Sigma)\subseteq T.
\]
Taking $\Sigma=T=\Qbar$ gives the phenomenon envisioned by Weierstrass
\cite{stackel1895,waldschmidt2003}. St\"ackel later obtained stronger
arithmetic interpolation results involving both a function and its inverse
\cite{stackel1902}; see also
\cite{huangmarquesmereb2010, vanderpoorten1968,waldschmidt2003}.

These results reveal the remarkable flexibility of analytic interpolation
when no restrictions are imposed on the Taylor coefficients. The situation
changes sharply when the coefficients themselves are required to satisfy
arithmetic constraints. This tension between arithmetic values and rigid
Taylor expansions motivated three problems formulated by Mahler, usually
called Problems~A, B, and C. Problems~B and~C have been solved; see, for
example, \cite{marquesmoreira2017,marquesmoreira2018}. The first problem,
which places the strongest direct restriction on the coefficients, remains
open.

Let $\Z\{z\}$ denote the ring of power series with integer coefficients that
are analytic in the open unit disk $\B$.

\begin{questionA}\label{question:mahlerA}
Does there exist a function
\[
        f(z)=\sum_{k\geq0}a_kz^k\in\Z\{z\},
\]
transcendental over $\C(z)$, whose coefficients are bounded and such that
\[
        f\bigl(\Qbar\cap\B\bigr)\subseteq\Qbar?
\]
\end{questionA}

Since a bounded sequence of integers takes values in a finite alphabet,
Problem~A asks whether arithmetic interpolation at every algebraic point of
the disk can coexist with an exceptionally rigid Taylor expansion. Related
directions include Mahler's work on lacunary series \cite{mahler1965}, the
asymptotic variants studied in \cite{franciscomarques2023, franciscomarques}, and recent
refinements concerning values of generalized Liouville series
\cite{bilumarquesmoreira2025}.

A result especially relevant to the present work was obtained by Marques and
Moreira in \cite{marquesmoreira2021}. They showed that boundedness can be
replaced by the requirement that every nonzero coefficient have no prime
divisors other than $2$ and $3$, equivalently that it belong to
\[
        \{\pm2^a3^b:a,b\geq0\}.
\]
This raises a natural question: are two multiplicatively independent
generators essential, or can one work with coefficient sizes generated by a
single number?

The local-annihilation mechanism of \cite{marquesmoreira2021} relies on two
multiplicatively independent prime factors and does not extend directly to a
single real generator. Indeed, Section~\ref{section:real-phases} shows that,
for a prime $p$, the only nonzero rational zeros of polynomials with
coefficients in $\{0\}\cup\{\pm p^m:m\geq0\}$ are the signed integral powers
of $p$. This obstruction motivates the use of algebraic complex phases,
where the \textit{phase} of a nonzero complex number $z$ is the unit-modulus factor
$z/|z|$.

The main idea of this paper is to separate coefficient size from coefficient
phase. Fix a real algebraic number $\lambda>1$ and define
\begin{equation}\label{eq:Dlambda}
        \D_\lambda
        :=\{0\}\cup
        \bigl\{\gamma\in\Qbar:\abs{\gamma}=\lambda^m
        \text{ for some }m\in\Z_{\geq0}\bigr\}.
\end{equation}
The permitted nonzero moduli still form the rank-one multiplicative
semigroup generated by $\lambda$, but algebraic complex phases remain
available. This Archimedean freedom supplies precisely the local
cancellations that are unavailable over a real one-generator alphabet.

Our main theorem shows that this phase freedom is sufficient not only to
control the values of the function, but also to make its full analytic jet
algebraic at every algebraic point of the disk.

\begin{theorem}\label{theorem:main}
Let $\lambda\in\Qbar\cap\R$ with $\lambda>1$. There are continuum many functions
\[
        f(z)=\sum_{k\geq0}a_kz^k\in\Qbar[[z]],
\]
analytic in $\B$ and having radius of convergence exactly $1$, whose
coefficients satisfy $a_k\in\D_\lambda$ for every $k\geq0$, and such that,
for every integer $s\geq0$,
\[
        f^{(s)}\bigl(\Qbar\cap\B\bigr)\subseteq\Qbar.
\]
Moreover, every derivative $f^{(s)}$ is transcendental over $\C(z)$.
\end{theorem}

Thus a rank-one restriction on coefficient moduli is compatible with the
algebraicity of the full analytic jet at every algebraic point. The theorem
does not solve Problem~A: the coefficients are generally nonreal and their
moduli are unbounded. Rather, it isolates the decisive role of algebraic
phases when coefficient sizes are confined to a rank-one semigroup.

The construction combines an algebraic polygon lemma, which produces local
annihilating multipliers, with a sparse placement of the resulting
polynomial blocks. The algebraicity of all derivatives is built into the
successive corrections, while transcendence follows from a gap-rigidity
principle for algebraic power series.

Throughout the paper, $\lambda>1$ denotes a fixed real algebraic number, and
$\D_\lambda$ is defined by \eqref{eq:Dlambda}. For a power series $F(z)$, we
write $[z^m]F(z)$ for the coefficient of $z^m$ in its Taylor expansion about
the origin.

\section{Auxiliary results}

In this section, we collect the auxiliary ingredients needed for the
construction and the proof of transcendence. These include an algebraic
polygonal cancellation principle, the construction of polynomial blocks
with increasing vanishing multiplicities, and a gap-propagation property
of algebraic power series.

\subsection{Algebraic polygonal cancellation}

We begin with the geometric ingredient that allows an arbitrary algebraic
base $\lambda>1$.

\begin{definition}
A tuple of positive real numbers $(\ell_0,\ldots,\ell_m)$, with $m\geq2$,
is called \emph{strictly polygon-admissible} if
\begin{equation}\label{eq:polygon-inequalities}
        \ell_j<\sum_{k\neq j}\ell_k
        \qquad(0\leq j\leq m).
\end{equation}
\end{definition}

The usual polygon criterion gives a closed nondegenerate polygon with these
side lengths. We need the additional arithmetic fact that algebraic lengths
may be realized by algebraic edge vectors.

\begin{lemma}[Algebraic polygon lemma]\label{lemma:algebraic-polygon}
Let $\ell_0,\ldots,\ell_m\in\Qbar\cap\R_{>0}$, where $m\geq2$, and suppose
that the tuple is strictly polygon-admissible. Then there exist nonzero
algebraic numbers $\xi_0,\ldots,\xi_m\in\Qbar$ such that
\[
        \abs{\xi_j}=\ell_j\quad(0\leq j\leq m),
        \qquad
        \xi_0+\cdots+\xi_m=0.
\]
Moreover, one may prescribe $\xi_0=\ell_0$.
\end{lemma}

\begin{proof}
We first record an algebraic splitting of a vector into two prescribed
lengths. Let $a,b,c\in\Qbar\cap\R_{>0}$ satisfy
\[
        \abs{a-b}<c<a+b,
\]
and let $w\in\Qbar$ with $\abs w=c$. Set
\[
        A:=\frac{a^2+c^2-b^2}{2c},
        \qquad
        B:=\sqrt{a^2-A^2}>0.
\]
The strict triangle inequalities imply $a^2-A^2>0$, so $A,B\in\Qbar\cap \R$.
The numbers
\[
        x_0:=A+iB,
        \qquad
        y_0:=c-x_0
\]
satisfy $x_0+y_0=c$, $\abs{x_0}=a$, and $\abs{y_0}=b$. Hence
\begin{equation}\label{eq:split-vector}
        x:=\frac{w}{c}x_0,
        \qquad
        y:=\frac{w}{c}y_0
\end{equation}
satisfy
\[
        x+y=w,
        \qquad
        \abs x=a,
        \qquad
        \abs y=b,
\]
and both $x$ and $y$ are algebraic.

We proceed by induction on $m+1$. If $m=2$, apply the preceding construction
with $a=\ell_1$, $b=\ell_2$, $c=\ell_0$, and $w=-\ell_0$. Then
$\xi_0:=\ell_0$ together with the resulting vectors $\xi_1,\xi_2$ has the
required properties.

Assume $m\geq3$ and that the statement holds for $m$ lengths. Put
\[
        a:=\ell_{m-1},
        \qquad
        b:=\ell_m,
        \qquad
        R:=\sum_{j=0}^{m-2}\ell_j,
        \qquad
        C:=\max_{0\leq j\leq m-2}\ell_j.
\]
Consider the interval
\begin{equation}\label{eq:delta-interval}
        \max\{\abs{a-b},\,2C-R,\,0\}
        <\delta<
        \min\{a+b,\,R\}.
\end{equation}
It is nonempty. Indeed, the polygon inequality for the larger of $a$ and
$b$ gives $\abs{a-b}<R$, while $\abs{a-b}<a+b$ is immediate. The polygon
inequality for $C$ gives $2C-R<a+b$. Moreover, at least two positive lengths
occur in the sum defining $R$, so $C<R$ and therefore $2C-R<R$. We may thus
choose a positive rational number $\delta$ satisfying
\eqref{eq:delta-interval}.

The reduced tuple
\[
        (\ell_0,\ldots,\ell_{m-2},\delta)
\]
is strictly polygon-admissible. The inequality $\delta<R$ is part of
\eqref{eq:delta-interval}, and, for $0\leq j\leq m-2$,
\[
        2\ell_j-R\leq2C-R<\delta,
\]
which is equivalent to
\[
        \ell_j<(R-\ell_j)+\delta.
\]
By the induction hypothesis there exist algebraic vectors
$\eta_0,\ldots,\eta_{m-2},w$ with respective moduli
$\ell_0,\ldots,\ell_{m-2},\delta$, such that
\[
        \eta_0+\cdots+\eta_{m-2}+w=0,
        \qquad
        \eta_0=\ell_0.
\]
The inequalities $\abs{a-b}<\delta<a+b$ allow us to apply
\eqref{eq:split-vector} and write $w=x+y$ with algebraic $x,y$ satisfying
$\abs x=a$ and $\abs y=b$. Replacing $w$ by $x,y$ completes the induction.
\end{proof}

\begin{remark}\label{remark:number-of-sides}
Let $M:=\lceil\lambda\rceil$. If
\[
        1,r_1,\ldots,r_M\in[1,\lambda),
\]
then each entry is strictly smaller than $\lambda\leq M$, whereas the sum
of the remaining $M$ entries is at least $M$. Hence the tuple
$(1,r_1,\ldots,r_M)$ is strictly polygon-admissible, and
Lemma~\ref{lemma:algebraic-polygon} applies uniformly.
\end{remark}

We now turn polygonal closure into a coefficient-preserving local
annihilator.

\begin{lemma}[$\lambda$-geometric annihilating multiplier]\label{lemma:factor}
Let $M:=\lceil\lambda\rceil$. Let
\[
        H(z)=\sum_{k=0}^{d}c_kz^k\in\Qbar[z]
\]
satisfy $H(0)=1$ and $c_k\in\D_\lambda$ for every $k$. Let
$\alpha\in\Qbar$ with $0<\abs\alpha<1$. Then, for every integer $N>d$,
there exist $u_1,\ldots,u_M\in\Qbar$ such that
\begin{equation}\label{eq:factor-form}
        Q(z):=1+\sum_{j=1}^{M}u_jz^{jN}
\end{equation}
satisfies the following properties:
\begin{enumerate}[label=\textup{(\alph*)}]
\item $Q(\alpha)=0$;
\item $\abs{u_j}=\lambda^{m_j}$ for suitable integers $m_j\geq0$;
\item every coefficient of $H(z)Q(z)$ belongs to $\D_\lambda$.
\end{enumerate}
\end{lemma}

\begin{proof}
Write $\rho:=\abs\alpha$. Since complex conjugation preserves algebraicity,
\[
        \rho=\sqrt{\alpha\overline\alpha}\in\Qbar\cap(0,1).
\]
For each $j\in\{1,\ldots,M\}$, choose the least integer $m_j\geq0$ such that
\[
        r_j:=\lambda^{m_j}\rho^{jN}\geq1.
\]
Since $\rho^{jN}<1$, one has $m_j\geq1$, and minimality gives
\begin{equation}\label{eq:normalized-lengths}
        1\leq r_j<\lambda.
\end{equation}
By Remark~\ref{remark:number-of-sides}, the tuple
$(1,r_1,\ldots,r_M)$ is strictly polygon-admissible. Applying
Lemma~\ref{lemma:algebraic-polygon} with the first vector prescribed to be
$1$, we obtain $x_1,\ldots,x_M\in\Qbar$ such that
\begin{equation}\label{eq:polygon-cancellation}
        1+x_1+\cdots+x_M=0,
        \qquad
        \abs{x_j}=r_j.
\end{equation}
Set
\[
        u_j:=\frac{x_j}{\alpha^{jN}}.
\]
Then $u_j\in\Qbar$ and
\[
        \abs{u_j}=\frac{r_j}{\rho^{jN}}=\lambda^{m_j}.
\]
Moreover, \eqref{eq:polygon-cancellation} yields
\[
        Q(\alpha)
        =1+\sum_{j=1}^{M}u_j\alpha^{jN}
        =1+\sum_{j=1}^{M}x_j
        =0.
\]

With $u_0:=1$, one has
\[
        H(z)Q(z)=\sum_{j=0}^{M}u_jz^{jN}H(z).
\]
Because $N>d$, the intervals $[jN,jN+d]$, $0\leq j\leq M$, are pairwise
disjoint. Hence no addition of nonzero coefficients occurs. Every nonzero
coefficient of $HQ$ has the form $u_jc_k$. If
$\abs{c_k}=\lambda^{e_k}$, then
\[
        \abs{u_jc_k}=\lambda^{m_j+e_k}.
\]
Thus every coefficient of $HQ$ belongs to $\D_\lambda$.
\end{proof}

\subsection{Polynomial blocks with increasing multiplicities}

Fix an enumeration without repetitions
\begin{equation}\label{eq:enumeration}
        \{\alpha_1,\alpha_2,\ldots\}
        =\bigl(\Qbar\cap\B\bigr)\setminus\{0\}.
\end{equation}

\begin{proposition}\label{proposition:blocks}
There exists a sequence $(H_n)_{n\geq0}$ of polynomials in $\Qbar[z]$ such
that:
\begin{enumerate}[label=\textup{(\roman*)}]
\item $H_0=1$ and $H_n(0)=1$ for every $n\geq1$;
\item every coefficient of every $H_n$ belongs to $\D_\lambda$;
\item $H_{n-1}$ divides $H_n$ for every $n\geq1$;
\item for every $n\geq1$ and every $1\leq i\leq n$,
\begin{equation}\label{eq:multiplicity}
        \ord_{z=\alpha_i}H_n(z)\geq n-i+1.
\end{equation}
\end{enumerate}
\end{proposition}

\begin{proof}
We proceed by induction on $n$. For $n=0$, set $H_0:=1$. All the required
properties are then immediate.

Fix $n\geq1$ and suppose that $H_{n-1}$ has already been constructed and
satisfies the stated properties. We construct $H_n$ by means of an
auxiliary induction indexed by the points
\[
        \alpha_1,\ldots,\alpha_n.
\]

Set
\[
        G_{n,0}:=H_{n-1}.
\]
Suppose that, for some $j\in\{1,\ldots,n\}$, the polynomial
$G_{n,j-1}\in\Qbar[z]$ has already been constructed, satisfies
\[
        G_{n,j-1}(0)=1,
\]
and has all its coefficients in $\D_\lambda$. Choose an integer
\[
        N_{n,j}>\deg G_{n,j-1}.
\]
Applying Lemma~\ref{lemma:factor} to the pair
$(G_{n,j-1},\alpha_j)$ with exponent $N_{n,j}$, we obtain a polynomial
\[
        Q_{n,j}(z)
        =
        1+\sum_{\nu=1}^{M}
        u_{n,j,\nu}z^{\nu N_{n,j}}
        \in\Qbar[z]
\]
such that
\[
        Q_{n,j}(\alpha_j)=0,
        \qquad
        Q_{n,j}(0)=1,
\]
every nonzero coefficient of $Q_{n,j}$ has modulus equal to a
nonnegative integral power of $\lambda$, and every coefficient of
\[
        G_{n,j}(z)
        :=
        G_{n,j-1}(z)Q_{n,j}(z)
\]
belongs to $\D_\lambda$.

The inequality
\[
        N_{n,j}>\deg G_{n,j-1}
\]
ensures that the shifted copies of $G_{n,j-1}$ occurring in the product
$G_{n,j-1}Q_{n,j}$ have pairwise disjoint supports. Consequently, no
addition of distinct nonzero coefficients occurs, which is precisely the
mechanism that preserves the coefficient condition in
Lemma~\ref{lemma:factor}. Moreover,
\[
        G_{n,j}(0)
        =
        G_{n,j-1}(0)Q_{n,j}(0)
        =
        1.
\]
Thus the auxiliary induction can be continued through
$j=1,\ldots,n$.

After completing the $n$ auxiliary steps, define
\[
        H_n:=G_{n,n}.
\]
By construction,
\[
        H_n
        =
        H_{n-1}\prod_{j=1}^{n}Q_{n,j}.
\]
It follows immediately that $H_{n-1}$ divides $H_n$. Furthermore, since
every factor $Q_{n,j}$ has constant term $1$,
\[
        H_n(0)
        =
        H_{n-1}(0)\prod_{j=1}^{n}Q_{n,j}(0)
        =
        1.
\]
The repeated application of Lemma~\ref{lemma:factor} also shows that every
coefficient of $H_n$ belongs to $\D_\lambda$.

It remains to prove the multiplicity estimate. Fix
$i\in\{1,\ldots,n\}$. Since
\[
        Q_{n,i}(\alpha_i)=0,
\]
we have
\[
        \ord_{z=\alpha_i}Q_{n,i}(z)\geq1.
\]
Every other factor has nonnegative order of vanishing at $\alpha_i$.
Therefore,
\begin{align*}
        \ord_{z=\alpha_i}H_n(z)
        &=
        \ord_{z=\alpha_i}H_{n-1}(z)
        +
        \sum_{j=1}^{n}
        \ord_{z=\alpha_i}Q_{n,j}(z)\\
        &\geq
        \ord_{z=\alpha_i}H_{n-1}(z)+1.
\end{align*}

If $i\leq n-1$, then the induction hypothesis gives
\[
        \ord_{z=\alpha_i}H_{n-1}(z)
        \geq
        (n-1)-i+1
        =
        n-i,
\]
and hence
\[
        \ord_{z=\alpha_i}H_n(z)
        \geq
        n-i+1.
\]
If $i=n$, the factor $Q_{n,n}$ alone gives
\[
        \ord_{z=\alpha_n}H_n(z)
        \geq1
        =
        n-n+1.
\]
Thus
\[
        \ord_{z=\alpha_i}H_n(z)\geq n-i+1
\]
for every $1\leq i\leq n$, which proves
\eqref{eq:multiplicity} and completes the induction.
\end{proof}

\subsection{Strongly lacunary power series}

As a last tool, we shall need the following classical theorem of Mahler
\cite[Chapter~II]{mahler1976}.

\begin{theorem}[Gap rigidity for algebraic power series]\label{theorem:mahler-lacunary}
Let
\[
        F(z)=\sum_{k\geq0}a_kz^k
\]
be a power series with positive radius of convergence. Suppose that there exist sequences of positive
integers $(s_n)$ and $(u_n)$ such that
\[
        s_n<u_n\leq s_{n+1},
\]
\[
       a_{s_n}a_{u_n}\neq 0, \quad a_k=0
        \qquad(s_n<k<u_n),
\]
and $u_n/s_n$ tends to infinity as $n\to \infty$. Then $F$ is transcendental over $\C(z)$.
\end{theorem}

\section{The proof of Theorem \ref{theorem:main}}

Let $(H_n)_{n\geq0}$ be supplied by Proposition~\ref{proposition:blocks},
and write
\[
        d_n:=\deg H_n
        \qquad(n\geq1).
\]
For $n\geq1$, put
\[
        r_n:=\frac{n}{n+1},
        \qquad
        M_n:=\max_{\abs z\leq r_n}\abs{H_n(z)}.
\]
Choose recursively a sequence of positive integers $(t_n)_{n\geq1}$ such that
\begin{equation}\label{eq:placement-convergence}
        r_n^{t_n}M_n\leq 2^{-n}
        \qquad(n\geq1),
\end{equation}
and
\begin{equation}\label{eq:block-separation}
        t_{n+1}\geq (n+1)(t_n+d_n)
        \qquad(n\geq1).
\end{equation}
Such a choice is possible. Indeed, since $0<r_n<1$, one has
\[
        r_n^tM_n\longrightarrow0
        \qquad(t\to\infty)
\]
for every fixed $n$. Thus, after choosing $t_1$ sufficiently large to
satisfy \eqref{eq:placement-convergence}, and having chosen $t_n$, we may
choose $t_{n+1}$ sufficiently large so that both
\eqref{eq:placement-convergence}, with $n+1$ in place of $n$, and
\eqref{eq:block-separation} are satisfied.

Set
\[
        \Omega:=\{-1,1\}^{\N}.
\]
For $\boldsymbol\varepsilon=(\varepsilon_1,\varepsilon_2,\ldots)\in\Omega$,
define
\begin{equation}\label{eq:function}
        f_{\boldsymbol\varepsilon}(z)
        :=\sum_{n=1}^{\infty}
        \varepsilon_nz^{t_n}H_n(z).
\end{equation}

\begin{proposition}\label{proposition:analytic}
For every $\boldsymbol\varepsilon\in\Omega$, the series
\eqref{eq:function} defines a function analytic in $\B$. Its Taylor
coefficients belong to $\D_\lambda$, and its radius of convergence is
exactly $1$.
\end{proposition}

\begin{proof}
Fix
\[
        \boldsymbol\varepsilon
        =(\varepsilon_1,\varepsilon_2,\ldots)\in\Omega.
\]
We first prove that \eqref{eq:function} converges locally uniformly in
$\B$. Let $R\in(0,1)$. Since $r_n\to1$, there exists $n_0=n_0(R)$ such
that $R\leq r_n$ for every $n\geq n_0$. Hence, if $\abs z\leq R$ and
$n\geq n_0$, then
\[
        \abs{\varepsilon_nz^{t_n}H_n(z)}
        \leq
        r_n^{t_n}\max_{\abs w\leq r_n}\abs{H_n(w)}
        =
        r_n^{t_n}M_n
        \leq2^{-n}.
\]
The Weierstrass $M$-test therefore gives uniform convergence on the closed
disk $\overline{B}(0,R)$. Since $R<1$ was arbitrary, the convergence is
uniform on every compact subset of $\B$, and the locally uniform limit
$f_{\boldsymbol\varepsilon}$ is analytic there.

Write
\[
        H_n(z)=\sum_{j=0}^{d_n}c_{n,j}z^j.
\]
The support of the $n$th shifted block is contained in
\[
        I_n:=[t_n,t_n+d_n].
\]
By \eqref{eq:block-separation},
\[
        \max I_n=t_n+d_n<t_{n+1}=\min I_{n+1},
\]
so the intervals $I_n$ are pairwise disjoint. Consequently, each Taylor
coefficient $a_k$ of $f_{\boldsymbol\varepsilon}$ is either zero or has the
form
\[
        a_k=\varepsilon_nc_{n,j}
\]
for uniquely determined $n\geq1$ and $0\leq j\leq d_n$ with
$k=t_n+j$. Proposition~\ref{proposition:blocks} gives
$c_{n,j}\in\D_\lambda$, and multiplication by a sign does not change its
modulus. Hence
\[
        a_k\in\D_\lambda
        \qquad(k\geq0).
\]

Since $H_n(0)=1$, the coefficient at the first exponent of the $n$th block
is
\begin{equation}\label{eq:marker}
        [z^{t_n}]f_{\boldsymbol\varepsilon}
        =
        \varepsilon_n.
\end{equation}
The separation condition implies that $t_n\to\infty$. Thus, writing
$f_{\boldsymbol\varepsilon}(z)=\sum_{k\geq0}a_kz^k$, we have
$\abs{a_{t_n}}=1$ for every $n$, and therefore
\[
        \limsup_{k\to\infty}\abs{a_k}^{1/k}
        \geq
        \limsup_{n\to\infty}\abs{a_{t_n}}^{1/t_n}
        =1.
\]
If $\mathcal R$ denotes the radius of convergence, the
Cauchy--Hadamard formula yields $\mathcal R\leq1$. On the other hand, the
local uniform convergence in $\B$ gives $\mathcal R\geq1$. Hence
$\mathcal R=1$.
\end{proof}

\begin{proposition}[Algebraic jets]\label{proposition:values}
For every integer $s\geq0$ and every
$\boldsymbol\varepsilon\in\Omega$,
\[
        f_{\boldsymbol\varepsilon}^{(s)}
        \bigl(\Qbar\cap\B\bigr)
        \subseteq\Qbar.
\]
\end{proposition}

\begin{proof}
Local uniform convergence permits termwise differentiation of
\eqref{eq:function}. Fix $s\geq0$ and a nonzero algebraic point
$\alpha_i\in\B$. By Proposition~\ref{proposition:blocks}, for every
$n\geq i+s$,
\[
        \ord_{z=\alpha_i}H_n(z)
        \geq n-i+1
        \geq s+1.
\]
Since $\alpha_i\neq0$, multiplication by $z^{t_n}$ does not change the
order of vanishing at $\alpha_i$. Therefore
\[
        \left.
        \frac{d^s}{dz^s}\bigl(z^{t_n}H_n(z)\bigr)
        \right|_{z=\alpha_i}
        =0
        \qquad(n\geq i+s).
\]
It follows that
\begin{equation}\label{eq:finite-jet-sum}
        f_{\boldsymbol\varepsilon}^{(s)}(\alpha_i)
        =
        \sum_{n=1}^{i+s-1}
        \varepsilon_n
        \left.
        \frac{d^s}{dz^s}\bigl(z^{t_n}H_n(z)\bigr)
        \right|_{z=\alpha_i},
\end{equation}
where an empty sum is interpreted as zero. The right-hand side is a finite
sum of algebraic numbers and is therefore algebraic.

At the origin,
\[
        f_{\boldsymbol\varepsilon}^{(s)}(0)
        =
        s!\,[z^s]f_{\boldsymbol\varepsilon}
        \in\Qbar.
\]
\end{proof}

\begin{proposition}\label{proposition:transcendence}
For every integer $s\geq 0$ and every
$\boldsymbol{\varepsilon}\in\Omega$, the function
$f_{\boldsymbol{\varepsilon}}^{(s)}$ is transcendental over $\C(z)$.
\end{proposition}

\begin{proof}
Write
\[
f_{\boldsymbol{\varepsilon}}(z)
=
\sum_{k\geq 0}a_kz^k
\]
and, for every $n\geq 1$, write
\[
H_n(z)
=
\sum_{j=0}^{d_n}c_{n,j}z^j.
\]
Since $d_n=\deg H_n$, we have $c_{n,d_n}\neq 0$, while
$H_n(0)=1$ gives $c_{n,0}=1$.

Fix an integer $s\geq 0$. Since $f_{\boldsymbol{\varepsilon}}$ is
analytic in $B(0,1)$, so is $f_{\boldsymbol{\varepsilon}}^{(s)}$; in
particular, its Taylor series at the origin has positive radius of
convergence. Since $t_n+d_n\to\infty$, choose $n_0\geq 1$ such that
\[
t_n+d_n>s
\]
for every $n\geq n_0$. For each $n\geq n_0$, define
\[
S_n:=t_n+d_n-s
\qquad\text{and}\qquad
U_n:=t_{n+1}-s.
\]
By \eqref{eq:block-separation}, we have
\[
t_{n+1}>t_n+d_n>s,
\]
so $S_n$ and $U_n$ are positive integers satisfying $S_n<U_n$. The $n$th shifted block is supported in
\[
[t_n,t_n+d_n],
\]
whereas the $(n+1)$st shifted block begins at $t_{n+1}$. Hence
\[
a_k=0
\qquad
(t_n+d_n<k<t_{n+1}).
\]
Using
\[
[z^m]f_{\boldsymbol{\varepsilon}}^{(s)}(z)
=
\frac{(m+s)!}{m!}\,a_{m+s},
\]
we therefore obtain
\[
[z^m]f_{\boldsymbol{\varepsilon}}^{(s)}(z)=0
\qquad
(S_n<m<U_n).
\]

We next verify that the coefficients at both endpoints of this gap are
nonzero. Since the coefficient of $z^{t_n+d_n}$ in the $n$th block is
$\varepsilon_n c_{n,d_n}$ and the shifted blocks have disjoint supports,
we have
\[
a_{t_n+d_n}
=
\varepsilon_n c_{n,d_n}
\neq 0.
\]
Consequently,
\[
[z^{S_n}]f_{\boldsymbol{\varepsilon}}^{(s)}(z)
=
\frac{(t_n+d_n)!}{(t_n+d_n-s)!}\,
\varepsilon_n c_{n,d_n}
\neq 0.
\]
Similarly, since $H_{n+1}(0)=1$, the coefficient of $z^{t_{n+1}}$ in
$f_{\boldsymbol{\varepsilon}}$ is $\varepsilon_{n+1}$, and hence
\[
[z^{U_n}]f_{\boldsymbol{\varepsilon}}^{(s)}(z)
=
\frac{t_{n+1}!}{(t_{n+1}-s)!}\,
\varepsilon_{n+1}
\neq 0.
\]

Moreover, \eqref{eq:block-separation} yields
\[
\frac{U_n}{S_n}
=
\frac{t_{n+1}-s}{t_n+d_n-s}
\longrightarrow\infty.
\]
Indeed,
\[
\frac{t_{n+1}-s}{t_n+d_n-s}
=
\frac{t_{n+1}}{t_n+d_n}
\,
\frac{1-s/t_{n+1}}{1-s/(t_n+d_n)},
\]
where the first factor tends to infinity and the second tends to $1$.
Finally,
\[
U_n
=
t_{n+1}-s
\leq
t_{n+1}+d_{n+1}-s
=
S_{n+1}.
\]

Thus the coefficient sequence of
$f_{\boldsymbol{\varepsilon}}^{(s)}$ satisfies all the hypotheses of
Theorem~\ref{theorem:mahler-lacunary}, with gap endpoints $S_n$ and
$U_n$. In particular, $f_{\boldsymbol{\varepsilon}}^{(s)}$ is a
nonpolynomial strongly lacunary power series. Theorem
\ref{theorem:mahler-lacunary} therefore implies that
$f_{\boldsymbol{\varepsilon}}^{(s)}$ is transcendental over $\C(z)$.
\end{proof}

\begin{proof}[Proof of Theorem~\ref{theorem:main}]
Propositions~\ref{proposition:analytic},
\ref{proposition:values}, and
\ref{proposition:transcendence} establish all the asserted analytic,
arithmetic, and transcendence properties.

If $\boldsymbol\varepsilon\neq\boldsymbol\eta$, then the two sign sequences
differ at some index $n$. By \eqref{eq:marker}, the coefficients of
$z^{t_n}$ in $f_{\boldsymbol\varepsilon}$ and
$f_{\boldsymbol\eta}$ are, respectively, $\varepsilon_n$ and $\eta_n$,
and are therefore distinct. Hence the map
\[
        \boldsymbol\varepsilon
        \longmapsto
        f_{\boldsymbol\varepsilon}
\]
is injective. Since
\[
        \abs{\Omega}
        =
        \abs{\{-1,1\}^{\N}}
        =
        2^{\aleph_0},
\]
the construction yields continuum many distinct functions with all the
stated properties.
\end{proof}

\section{A non-Archimedean obstruction for real phases}
\label{section:real-phases}

For a real algebraic number $\lambda>1$, define the real rank-one alphabet
\begin{equation}\label{eq:Rlambda}
        \Rset_\lambda
        :=
        \{0\}\cup
        \{\pm\lambda^m:m\in\Z_{\geq0}\}.
\end{equation}
The construction used in the proof of Theorem~\ref{theorem:main} relies on
algebraic complex phases to produce local cancellations while preserving the
prescribed coefficient moduli. We now show that this phase freedom overcomes
a genuine arithmetic obstruction: zeros of polynomials with coefficients in
$\Rset_\lambda$ are necessarily subject to strong non-Archimedean
restrictions.

\begin{proposition}\label{proposition:real-obstruction}
Let $\lambda\in\Qbar^\times$, let $K$ be a number field containing
$\lambda$, and let
\[
        P(z)=\sum_{j=0}^{d}c_jz^j\in K[z]
\]
be a nonzero polynomial whose nonzero coefficients belong to
\[
        \{\pm\lambda^m:m\in\Z_{\geq0}\}.
\]
Suppose that $\alpha\in\Qbar^\times$ satisfies $P(\alpha)=0$, and enlarge
$K$, if necessary, so that $\alpha\in K$. Then
\[
        v(\alpha)=0
\]
for every finite place $v$ of $K$ such that $v(\lambda)=0$.

Equivalently, if
\[
        S_\lambda
        :=
        \{v:v \text{ is a finite place of }K
        \text{ and }v(\lambda)\neq0\},
\]
then every nonzero zero of $P$ is an $S_\lambda$-unit.
\end{proposition}

\begin{proof}
Since $\alpha\neq0$ and $P(\alpha)=0$, the polynomial $P$ must contain at
least two nonzero terms. Let
\[
        j_0<j_1<\cdots<j_r,
        \qquad r\geq1,
\]
be the exponents corresponding to its nonzero coefficients. Thus
\[
        P(z)=\sum_{k=0}^{r}c_{j_k}z^{j_k}.
\]

Fix a finite place $v$ of $K$ such that $v(\lambda)=0$. Since every
$c_{j_k}$ is of the form $\pm\lambda^{m_k}$ for some $m_k\geq0$, one has
\[
        v(c_{j_k})=0
        \qquad(0\leq k\leq r).
\]
Consequently,
\[
        v\bigl(c_{j_k}\alpha^{j_k}\bigr)
        =
        j_kv(\alpha).
\]

Suppose first that $v(\alpha)>0$. Since
$j_0<j_k$ for every $k\geq1$, it follows that
\[
        v\bigl(c_{j_0}\alpha^{j_0}\bigr)
        <
        v\bigl(c_{j_k}\alpha^{j_k}\bigr)
        \qquad(1\leq k\leq r).
\]
Hence $c_{j_0}\alpha^{j_0}$ is the unique term of least valuation in the
vanishing sum
\[
        \sum_{k=0}^{r}c_{j_k}\alpha^{j_k}=0.
\]
This is impossible: in a non-Archimedean valued field, a vanishing finite
sum cannot have a unique term of strictly minimal valuation.

Similarly, suppose that $v(\alpha)<0$. Since $j_k<j_r$ for every
$k<r$, one obtains
\[
        v\bigl(c_{j_r}\alpha^{j_r}\bigr)
        <
        v\bigl(c_{j_k}\alpha^{j_k}\bigr)
        \qquad(0\leq k<r).
\]
Thus $c_{j_r}\alpha^{j_r}$ is the unique term of least valuation in the
same vanishing sum, which is again impossible. Therefore
\[
        v(\alpha)=0.
\]

Since this holds at every finite place outside $S_\lambda$, the element
$\alpha$ is an $S_\lambda$-unit.
\end{proof}

The preceding proposition becomes especially explicit when the generator is
a rational prime.

\begin{corollary}\label{corollary:prime-real-obstruction}
Let $p$ be a prime. A nonzero rational number $\alpha$ is a zero of a
nonzero polynomial whose nonzero coefficients belong to
\[
        \{\pm p^m:m\in\Z_{\geq0}\}
\]
if and only if
\[
        \alpha=\pm p^k
        \qquad\text{for some }k\in\Z.
\]
\end{corollary}

\begin{proof}
Suppose that such a polynomial vanishes at
\[
        \alpha=\frac{a}{q}\in\Q^\times,
\]
where $a$ and $q$ are coprime integers and $q>0$. For every rational prime
$\ell\neq p$, Proposition~\ref{proposition:real-obstruction} gives
\[
        v_\ell(\alpha)=0.
\]
Hence no prime other than $p$ divides either $a$ or $q$. Since $a$ and $q$
are coprime, one of them is equal to $1$ up to sign, and therefore
\[
        \alpha=\pm p^k
\]
for some $k\in\Z$.

Conversely, if $k\geq0$, then $\pm p^k$ is a zero of
\[
        z\mp p^k,
\]
whereas, if $r\geq1$, then $\pm p^{-r}$ is a zero of
\[
        p^rz\mp1.
\]
In each case, all nonzero coefficients belong to
$\{\pm p^m:m\in\Z_{\geq0}\}$.
\end{proof}

We next record the consequence that is most directly relevant to the
local-annihilation procedure.

\begin{corollary}\label{corollary:infinitely-many-obstructed-points}
For every fixed real algebraic number $\lambda>1$, there exist infinitely
many rational numbers $\alpha\in(0,1)$ that are not zeros of any nonzero
polynomial with coefficients in $\Rset_\lambda$.
\end{corollary}

\begin{proof}
Let $K$ be a number field containing $\lambda$, and let $S$ be the finite
set of rational primes lying below the finite places $v$ of $K$ for which
\[
        v(\lambda)\neq0.
\]
Choose a rational prime $q\notin S$ and set
\[
        \alpha:=\frac1q.
\]
For every finite place $v$ of $K$ lying above $q$, one has
\[
        v(\lambda)=0
        \qquad\text{and}\qquad
        v(\alpha)<0.
\]
Proposition~\ref{proposition:real-obstruction} therefore shows that
$\alpha$ cannot be a zero of any nonzero polynomial with coefficients in
$\Rset_\lambda$. Since there are infinitely many primes outside the finite
set $S$, the conclusion follows.
\end{proof}

Proposition~\ref{proposition:real-obstruction} explains why algebraic
complex phases enter the proof of Theorem~\ref{theorem:main}. They are not
introduced merely as a technical convenience: they provide local
cancellations at arbitrary algebraic points that are unavailable, in
general, when all coefficients are restricted to the real rank-one alphabet
$\Rset_\lambda$.

The proposition does not rule out the possibility that a fundamentally
different global construction might succeed with coefficients in
$\Rset_\lambda$. It shows, more precisely, that the coefficient-preserving
local-annihilation mechanism developed in this paper cannot be transferred
unchanged to the real-phase setting.

Theorem~\ref{theorem:main} is therefore complementary to the integral
construction of \cite{marquesmoreira2021}. In that work, the coefficients
are real integers whose prime divisors are restricted to $2$ and $3$. In
the present construction, the nonzero coefficient moduli belong to the
rank-one semigroup generated by $\lambda$, while algebraic complex phases
supply the flexibility needed for local cancellation. Neither result
settles Mahler's Problem~A for bounded integer coefficients.

\section*{Acknowledgements}

The author thanks the Conselho Nacional de Desenvolvimento Cient\'ifico e
Tecnol\'ogico (CNPq) for its financial support.

\section*{Declaration of competing interest}

The author declares that he has no known competing financial interests or
personal relationships that could have appeared to influence the work
reported in this paper.

\section*{Declaration of generative AI and AI-assisted technologies
in the manuscript preparation process}

During the preparation of this work, the author used ChatGPT (OpenAI) solely for English grammar correction and language editing. The author reviewed the resulting text and takes full responsibility for the content of the article.

\end{document}